\documentclass[11pt,leqno]{amsart}%
\usepackage{amsmath}
\usepackage{amsfonts}
\usepackage{amssymb}
\usepackage{graphicx}%
\providecommand{\U}[1]{\protect\rule{.1in}{.1in}}

\newtheorem{theorem}{Theorem}

\newtheorem{corollary}[theorem]{Corollary}

\newtheorem*{problem*}{Problem 1}
\newtheorem{proposition}[theorem]{Proposition}
\newtheorem{remark}[theorem]{Remark}

\begin{document}

\title{Remarks on $L^{p}$-vanishing results in geometric analysis}

\date{\today}

\author{Stefano Pigola}
\address{Dipartimento di Fisica e Matematica\\
Universit\`a dell'Insubria - Como\\
via Valleggio 11\\
I-22100 Como, Italy.}
\email{stefano.pigola@uninsubria.it}

\author{Giona Veronelli}
\address{D\'epartement de Math\'ematiques
Universit\'e de Cergy-Pontoise\\
Site de Saint Martin
2, avenue Adolphe Chauvin
95302 Cergy-Pontoise Cedex
France}
\email{giona.veronelli@gmail.com}



\begin{abstract}
We survey some $L^{p}$-vanishing results for solutions of Bochner or Simons type
equations with refined Kato inequalities, under spectral assumptions on the
relevant Schr\"{o}dinger operators. New aspects are included in the picture. In particular,
an abstract version of a structure theorem for stable minimal hypersurfaces
of finite total curvature is observed. Further geometric applications are discussed.

\end{abstract}

\maketitle

\tableofcontents
\section{Introduction and some vanishing results}

This paper originates from an attempt to understand the abstract content of a
structure theorem for stable minimal hypersurfaces of finite total scalar curvature.

\begin{theorem}
\label{th_CSZ}Let $f:M^{m}\rightarrow\mathbb{R}^{N}$ be a complete,
$m$-dimensional, minimal submanifold, $m\geq2$. Assume that its second
fundamental form $\mathrm{II}$ satisfies

\begin{itemize}
\item[(i)] $\left\vert \mathrm{II}\right\vert ^{2}\in L^{m/2}\left(  M\right)
$

\item[(ii)] The \textquotedblleft stability operator\textquotedblright%
\ $\mathcal{L}=-\Delta-\left\vert \mathrm{II}\right\vert ^{2}$ has
non-negative spectrum.
\end{itemize}

\noindent Then, $\left\vert \mathrm{II}\right\vert \equiv0$, that is,
$f\left(  M\right)  $ is an affine $m$-plane .
\end{theorem}

Partial forms of this beautiful result, under various restrictions, have been
obtained by R. Schoen, L. Simon and S.T. Yau, \cite{SSY-Acta}, M.P. do Carmo
and C.K. Peng, \cite{DCP-BAMS}, \cite{DCP-Beijing} and P. B\'{e}rard,
\cite{Berard2-Remarque}. In the above general form, it is due to M. Anderson,
\cite{And-preprint}, and, Y.-B. Shen and X.-H. Zhu, \cite{SZ-AmerJour}.

In case $m=2$, it is known from previous work by Bernstein (minimal graphs),
do Carmo-Peng, \cite{DCP-BAMS}, and D. Fisher-Colbrie and Schoen (general
surfaces), \cite{FCS-CPAM}, that condition (i) is unnecessary. Actually, it
has been conjectured that this conclusion can be extended to every dimension
$m\leq7$. The upper bound $m=7$ is due to the famous non-planar (and stable)
minimal graphs by E. Bombieri, E. de Giorgi and E. Giusti. We also record
that, in higher dimensions, condition $(i)$ can be replaced by the next more
general requirement, \cite{XD-Basel}, \cite{PV-DJA},

\begin{itemize}
\item[(i)'] $\left\vert \mathrm{II}\right\vert ^{2}\in L^{\gamma}\left(
M\right)  $, for some $\gamma\geq m/2$.
\end{itemize}

As we shall see in the next section, the approaches proposed by the above
mentioned authors are very different from each others. However they have a key
analytic ingredient in common, i.e., the fact that the second fundamental form
satisfies the Simons equation%
\begin{equation}
\frac{1}{2}\Delta\left\vert \mathrm{II}\right\vert ^{2}+\left\vert
\mathrm{II}\right\vert ^{2}\left\vert \mathrm{II}\right\vert ^{2}=\left\vert
D\mathrm{II}\right\vert ^{2}\geq\left(  1+\frac{2}{m}\right)  \left\vert
\nabla\left\vert \mathrm{II}\right\vert \right\vert ^{2}. \label{simons}%
\end{equation}
The last inequality is a special case of what is known in the literature as a
\textquotedblleft refined Kato inequality\textquotedblright, \cite{Br}, \cite{CGH}. In particular,
from (\ref{simons}) it follows that%
\begin{equation}
\left\vert \mathrm{II}\right\vert \left(  \Delta\left\vert \mathrm{II}%
\right\vert +\left\vert \mathrm{II}\right\vert ^{2}\left\vert \mathrm{II}%
\right\vert \right)  -\frac{2}{m}\left\vert \nabla\left\vert \mathrm{II}%
\right\vert \right\vert ^{2}\geq0, \label{simons1}%
\end{equation}
weakly on $M$.

From the analytic viewpoint, vanishing results for $L^{p}$-solutions of PDEs,
under spectral assumptions on the relevant Schr\"{o}dinger operator, are quite
well understood, \cite{ER-ActaApplMath}, \cite{Berard1-Manuscripta},
\cite{PRS-JFA}, \cite{PRS-Book}. For instance, consider the following
statement from \cite{PRS-JFA}.

\begin{theorem}
\label{th_vanishing1}Let $\left(  M,\left\langle ,\right\rangle \right)  $ be
a complete Riemannian manifold and let $0\leq\psi\in Lip_{loc}\left(
M\right)  $ be a distributional solution of the differential inequality%
\begin{equation}
\psi\Delta\psi+a\left(  x\right)  \psi^{2}+A\left\vert \nabla\psi\right\vert
^{2}\geq0, \label{inequality}%
\end{equation}
for some $a\left(  x\right)  \in C^{0}\left(  M\right)  $ and $A\in\mathbb{R}%
$. Assume that the bottom of the spectrum of the (modified) Schr\"{o}dinger
operator $\mathcal{L}_{H}=-\Delta-Ha\left(  x\right)  $ satisfies%
\begin{equation}
\lambda_{1}\left(  \mathcal{L}_{H}\right)  \geq0, \label{spectral}%
\end{equation}
for some $H\geq A+1>0$. If%
\begin{equation}
\int^{+\infty}\frac{dR}{\int_{\partial B_{R}\left(  o\right)  }\psi^{2p}%
}=+\infty, \label{integral}%
\end{equation}
for some origin $o\in M$ and for some $A+1\leq p\leq H$, then the following hold:

\begin{enumerate}
\item[(i)] If $H\geq p>A+1$ then $\psi\equiv\mathrm{const.}$ and either
$a\left(  x\right)  \equiv0$ or $\psi\equiv0$.

\item[(ii)] If $H=A+1$ then $\psi$ satisfies (\ref{inequality}) with the
equality sign.
\end{enumerate}
\end{theorem}

\begin{remark}
For future purposes, and for comparison with other vanishing results, we point
out the following chain of (strict) implications:%
\[
\psi\in L^{2p}\Rightarrow\int_{B_{R}}\psi^{2p}\leq CR^{2}\Rightarrow
\int^{+\infty}\frac{RdR}{\int_{B_{R}}\psi^{2p}}=+\infty\Rightarrow
\text{(\ref{integral}).}%
\]

\end{remark}

Thus, in the standard \textquotedblleft stable\textquotedblright\ case $H=1$,
one can get triviality provided a small power (say $\leq2$) of $\psi$ is
integrable. It is important to observe that large powers of $\psi$ (like $p=m/2$
in the geometric case) are not allowed.

\begin{proof}
[Proof of Theorem \ref{th_vanishing1} (sketch)]The method to prove Theorem
\ref{th_vanishing1} resembles the one used to prove the generalized maximum
principle for the operator $\mathcal{L}_{H},$ \cite{PW-Book}. Namely, following
\cite{FCS-CPAM}, one interprets (\ref{spectral}) as the existence of a smooth
solution $\varphi>0$ of the linear equation%
\[
\mathcal{L}_{H}\varphi=0.
\]
Next, we combine $\psi$ and $\varphi$ to obtain a new function%
\[
u=\frac{\psi^{p}}{\varphi^{p/H}}%
\]
which is subharmonic with respect to a suitable diffusion-type operator, i.e.,%
\[
\Delta_{\omega}u:=\omega^{-1}\operatorname{div}\left(  \omega\nabla u\right)
\geq0,
\]
with $\omega=\varphi^{2p/H}$. Assume $\psi\in L^{2p}\left(  M\right)  ,$ the
more general case where (\ref{integral}) holds true can be dealt similarly.
Since $u\in L^{2}\left(  M,\omega d\mathrm{vol}\right)  $, a version (for
diffusion operators) of the classical theorem by Yau on positive, $L^{q>1}%
$-subharmonic functions implies that $u$ is constant. The remaining
conclusions on $\psi$ now follows easily from considerations on the constants
$A,H,p$.
\end{proof}

It is worth to point out that, actually, one can obtain triviality making a
direct use of the spectral assumption with suitable test functions. This
method, which was used in \cite{SSY-Acta}, \cite{Berard1-Manuscripta}, works
particularly well in non-linear contexts, and vanishing results for the energy
density of a $p$-harmonic map are nontrivial examples, \cite{PV-GeomDedicata}.
The \textquotedblleft direct\textquotedblright\ method has a further nice
feature. It enables one to obtain vanishing for $\psi\in L^{2p}$ even if $p>H$. This is shown
in the following result that, in its full generality, has been never observed before.
\begin{theorem}
\label{th_direct}Let $\left(  M,\left\langle ,\right\rangle \right)  $ be a
complete Riemannian manifold and let $0\leq\psi\in Lip_{loc}\left(  M\right)
$ be a distributional solution of the differential inequality
(\ref{inequality}) for some $0\leq a\left(  x\right)  \in C^{0}\left(  M\right)  $
and $A\in\mathbb{R}$. Assume that the bottom of the spectrum of the (modified)
Schr\"{o}dinger operator $\mathcal{L}_{H}=-\Delta-Ha\left(  x\right)  $ satisfies
(\ref{spectral}) for some $H\geq A+1>0$. Let $p_{0}\leq H\leq p_{1}$
be the solutions of%
\[
q^{2}-2Hq+H\left(  A+1\right)  =0.
\]
If%
\[
\int_{B_{R}\left(  o\right)  }\psi^{2p}=o\left(  R^{2}\right)  \text{, as
}R\rightarrow+\infty,
\]
for some origin $o\in M$ and for some%
\[
p_{0}\leq p \leq p_{1},
\]
then
\begin{itemize}
\item[(i)] If $H>A+1$ and $p_{0} < p < p_{1}$, then $\psi\equiv\mathrm{const.}$
\item[(ii)] If $H=A+1$ then $\psi$ satisfies (\ref{inequality}) with equality sign on the open set $\Omega$ where $\psi^A>0$. In particular, it holds on $M$ if $A\leq 0$.
\end{itemize}
\end{theorem}

Thus, as expected (and confirmed by computations), there is again an upper
bound for the integrability exponent and this is related to the (Kato)
constant $A$. For instance, in the minimal surface setting, this upper bound
prevented B\'{e}rard to reach dimensions $m\geq6$. Very recently, L.F. Tam and
D. Zhou employed this method to give characterizations of higher-dimensional
catenoids, \cite{TZ-TAMS}. It is also important to remark that
with the direct method we are able to reach low exponents $p_0<p \leq A+1$, i.e., even smaller than those
considered in Theorem \ref{th_vanishing1}. Apparently, one of the main drawbacks of the direct approach, when compared with Theorem \ref{th_vanishing1}, concerns with the case $H=A+1$ where one
obtains equality only on the set $\Omega$ where $\psi^A$ is positive.
However in geometric  applications this is not a serious problem, because, typically, $A$ arises from a refined Kato inequality and therefore it is negative.

\begin{proof}
[Proof of Theorem \ref{th_direct}.] The proof aims at obtaining a
Caccioppoli-type inequality for $\psi$. Define the distribution
\begin{align}\label{Einequality}
0\leq \mathcal{E}:=\psi\Delta\psi+a\psi^2+A|\nabla\psi|^2.
\end{align}
Note that since by assumption $\mathcal{E}$ is nonnegative, $\mathcal{E}$ is in fact a positive measure. By the weak formulation of inequality \eqref{Einequality},
\begin{align}\label{d_inequality}
0\leq \int \eta d\mathcal{E} = \int_M a\eta\psi^2 + \int_MA\eta|\nabla\psi|^2-\int_M \left\langle \nabla(\eta\psi),\nabla\psi \right\rangle,
\end{align}
for each test function $0\leq \eta\in Lip_c\left( M \right)$. From now on, until otherwise specified, we assume that $0<p<1$. The case $p\geq1$ is easier and will be dealt with at the end of proof.
We apply \eqref{d_inequality} to the function
$\eta = \rho^{2}\psi_\delta^{2p-2}$ where $0 \leq \rho\in C_{c}^{\infty}\left( M \right)$ is  a cut-off function to be chosen later, and, for every $\delta>0$, $\psi_\delta=\psi+\delta>0$, to get
\begin{align}\label{di}
0&\leq \int_M \rho^{2}\psi_\delta^{2p-2}d\mathcal{E}\\
\nonumber  &\leq \int_M a\rho^2\psi_\delta^{2p-2}\psi^2 + \int_M A\rho^2\psi_\delta^{2p-2}|\nabla\psi|^2 - (2p-2)\int_M\rho^2\psi_\delta^{2p-3}\psi|\nabla\psi|^2\\
&\nonumber-\int_M\rho^2\psi_\delta^{2p-2}|\nabla\psi|^2-2\int_M\rho\psi_\delta^{2p-2}\psi\left\langle \nabla\rho,\nabla\psi\right\rangle.
\end{align}
On the other hand, the spectral assumption on $\mathcal{L}_H$ implies that, for each test function $\phi\in Lip_c(M)$,
\[
\int_M |\nabla\phi|^2\geq\int_MHa\phi^2.
\]
Choosing $\phi=\rho\psi_\delta^p$, we get
\begin{align*}
\int_M a\rho^2\psi_\delta^{2p-2}\psi^2&\leq\int_M a\rho^2\psi_\delta^{2p}\\
&\nonumber\leq \frac{p^2}H\int_M\rho^2\psi_\delta^{2p-2}|\nabla\psi|^2+\frac1H\int_M\psi_\delta^{2p}|\nabla\rho|^2\\
&\nonumber+\frac{2p}H\int_M\rho\psi_\delta^{2p-1}\left\langle \nabla\psi,\nabla\rho\right\rangle.
\end{align*}
Inserting this latter in \eqref{di}  gives
\begin{align*}
0&\leq \int_M \rho^{2}\psi_\delta^{2p-2} d\mathcal{E}\\
&\nonumber\leq \left(\frac{p^2}H+A-1\right)\int_M\rho^2\psi_\delta^{2p-2}|\nabla\psi|^2 - (2p-2)\int_M\rho^2\psi_\delta^{2p-3}\psi|\nabla\psi|^2 \\
&\nonumber+ \frac 1H \int_M\psi_\delta^{2p}|\nabla\rho|^2+2\int_M\rho\left(\frac pH\psi_\delta-\psi\right)\psi_\delta^{2p-2}\left\langle \nabla\rho,\nabla\psi\right\rangle.\nonumber
\end{align*}
Recalling that $0<p<1$, since $\psi<\psi+\delta=\psi_\delta$, from the above we get
\begin{align}\label{0cacc}
&\int_M \rho^{2}\psi_\delta^{2p-2} d\mathcal{E}+\left(2p - \frac {p^2}H - (A+1)\right)\int_M\rho^2\psi_\delta^{2p-2}|\nabla\psi|^2 \\&
\leq \frac 1H \int_M\psi_\delta^{2p}|\nabla\rho|^2+2\left(\frac pH-1\right)\int_M\rho\psi_\delta^{2p-1}\left\langle \nabla\rho,\nabla\psi\right\rangle\nonumber\\
&+2\delta\int_M\rho\psi_\delta^{2p-2}\left\langle \nabla\rho,\nabla\psi\right\rangle
.\nonumber
\end{align}
Now, assume $H>A+1$. Since $\delta\leq\psi_\delta$, $\mathcal E\geq 0$, recalling the Young inequality $2ab\leq\varepsilon a^{2}+\varepsilon
^{-1}b^{2}$, valid for each $a,b\geq0$ and $\varepsilon>0$, and applying Schwarz inequality, we deduce that%
\begin{align*}
&2\left(\frac pH-1\right)\int_M\rho\psi_\delta^{2p-1}\left\langle \nabla\rho,\nabla\psi\right\rangle+2\delta\int_M\rho\psi_\delta^{2p-2}\left\langle \nabla\rho,\nabla\psi\right\rangle \\&\leq \varepsilon \int_{M}\left\vert \nabla\psi\right\vert ^{2}\rho
^{2}\psi_\delta^{2p-2}
+ \frac 1\varepsilon\left(\left|\frac pH -1\right|+1\right)^2 \int_{M}\psi_\delta^{2p}\left\vert
\nabla\rho\right\vert ^{2}.
\end{align*}
Using this inequality into (\ref{0cacc}) we finally obtain
\begin{align}\label{cacc}
F\int_{M}\left\vert \nabla\psi\right\vert ^{2}\rho
^{2}\psi_\delta^{2p-2}\leq G\int_{M}\psi_\delta^{2p}\left\vert
\nabla\rho\right\vert ^{2},
\end{align}
where
\begin{align*}
&F:=2p - \frac {p^2}H - (A+1) -\varepsilon,\\
&G:= \frac 1H + \frac 1\varepsilon\left(\left|\frac pH -1\right|+1\right)^2.
\end{align*}
Note that, by assumption,
\[
p^{2}-2Hp+H\left(  A+1\right)  <0,
\]
therefore  we can take $\varepsilon$ small enough so that $F>0$.
Moreover, using monotone and dominated convergence we see that \eqref{cacc} holds even in the limit $\delta\to0$. Thus, if we choose $\rho=\rho_R\in C_{c}^{\infty}\left(  B_{R}\left(  o\right)
\right)  $ in such a way that $\left\vert \nabla\rho_R\right\vert \leq
R^{-1}$ and $\rho_R\equiv 1$ on $B_{R/2}\left(  o\right)  $, we get
\begin{align*}
F\int_{B_{R/2}}\left\vert \nabla\psi\right\vert ^{2}\psi^{2p-2}\leq \frac G{R^2}\int_{B_R}\psi^{2p},
\end{align*}
and letting
$R\rightarrow+\infty$ we conclude that $\psi$ is constant on the set
\[
\Omega_{p-1}=\left\{  x\in M:\psi^{2(p-1)}\left(  x\right)  >0\right\}.\]
Since $0<p<1$, then $\Omega_{p-1}=M$ and $\psi$ is constant on $M$. This completes the proof in the case $H>A+1$.\\
Now suppose that $H=A+1$. Since, in this case, $p_0=p_1 = A+1 =H$,  \eqref{0cacc} reads
\begin{align}\label{caccE}
0\leq \int_M \rho^{2}\psi_\delta^{2A} d\mathcal{E}
&\leq  \frac 1H \int_M\psi_\delta^{2A+2}|\nabla\rho|^2+2\delta\int_M\rho\psi_\delta^{2A}\left\langle \nabla\rho,\nabla\psi\right\rangle\\
&\leq  \frac 1H \int_M\psi_\delta^{2A+2}|\nabla\rho|^2+2\delta\int_M\rho\psi_\delta^{2A}|\nabla\rho||\nabla\psi|.\nonumber
\end{align}
We claim that
\begin{equation}
2\delta\int_{M}\rho\psi_{\delta}^{2A}\left\vert \nabla\rho\right\vert
\left\vert \nabla\psi\right\vert \rightarrow0\text{, as }\delta\rightarrow
0.\label{state}%
\end{equation}
Indeed, since $0<\delta\leq\psi_\delta$,  recalling also that we are in the case $-1<A=p-1<0$, we can compute
\begin{align}\label{reg}
2\int_{M}\delta\rho\psi_{\delta}^{2A}\left\vert \nabla\rho\right\vert
\left\vert \nabla\psi\right\vert  & =2\int_{M}\left\vert \nabla\rho\right\vert
\left\{  \delta\psi_{\delta}^{A}\right\}  \left\{  \rho\psi_{\delta}%
^{A}\left\vert \nabla\psi\right\vert \right\}  \\
&\leq2\delta^{A+1}\int_{M}  \left\vert \nabla\rho\right\vert
\left\{  \rho\psi_{\delta}%
^{A}\left\vert \nabla\psi\right\vert \right\}  \nonumber\\
& \leq \delta^{A+1}\left\{\int_{M}\left\vert
\nabla\rho\right\vert ^{2}+\int_{M}\rho^{2}\psi_\delta^{2A}\left\vert \nabla\psi
\right\vert ^{2}\right\}.\nonumber
\end{align}
Now, we recall from the regularity theory for Bochner
type inequalities in the presence of a refined Kato inequality (\cite{PRS-JFA} or Lemma 4.12 and
Lemma 4.13 in \cite{PRS-Book}) that%
\[
\psi^{A+1}\in W_{loc}^{1,2}\left(  M\right)
\]
with%
\[
\nabla\psi^{A+1}=\left(  A+1\right)  \psi^{A}\nabla\psi.
\]
Since
\[
\int_{M}\rho^{2}\psi_\delta^{2A}\left\vert \nabla\psi
\right\vert ^{2}\leq \int_{M}\rho^{2}\psi^{2A}\left\vert \nabla\psi
\right\vert ^{2} = \int_M\frac{\rho^2}{(A+1)^2}\left\vert \nabla\psi^{A+1}
\right\vert ^{2}<+\infty
\]
we deduce that the quantity
\[
\left\{\int_{M}\left\vert
\nabla\rho\right\vert ^{2}+\int_{M}\rho^{2}\psi_\delta^{2A}\left\vert \nabla\psi
\right\vert ^{2}\right\}
\]
on the right hand side of (\ref{reg}) is bounded uniformly in $\delta$. The validity of \eqref{state} now follows by letting $\delta\to0$ in \eqref{reg}, thus proving the claim. With this preparation, we can now take the limits as $\delta\to 0$ in (\ref{caccE}) and conclude
\begin{align}\label{caccE-}
0\leq \int_M \rho^{2}\psi^{2A} d\mathcal{E}\leq  \frac 1H \int_M\psi^{2A+2}|\nabla\rho|^2.
\end{align}
Then, choosing as before $\rho=\rho_R$ and letting $R\to\infty$ we obtain that $\mathcal{E}\equiv 0$ on the set $\Omega_A$ where $\psi^A>0$, thus completing the proof of the Theorem for $0<p<1$.\\

If $p\geq 1$, all the above arguments can be greatly simplified because one can obtain the validity of inequality (\ref{di}) with $\delta=0$. This is obviously true if $2p-2\geq1$, i.e., $p\geq3/2$. On the other hand, if $1\leq p <3/2$, then
\[
\rho^2\psi_\delta^{2p-3}\psi|\nabla\psi|^2 \leq \rho^2\psi^{2p-2}|\nabla\psi|^2 \in L^1.
\]
Accordingly, we can take the limits as $\delta\to0$ and use dominated convergence to deduce that inequality (\ref{di}) persists with $\delta=0$ and $\psi_\delta=\psi$, as claimed.
Now, in case $H>A+1$ we get that $\psi$ is constant on $\Omega_{p-1}$. Using a
connectedness argument, it follows easily that either $\Omega_{p-1}=\emptyset$ or
$\Omega_{p-1}=M$. In any case, $\psi$ is constant.\\
Finally, in case $H=A+1(\geq 1)$, we immediately get the validity of \eqref{caccE-} and the desired conclusion follows.
\end{proof}

What makes Theorem \ref{th_vanishing1} and  Theorem \ref{th_direct} interesting is that,
due to their abstract formulations, they can be adapted to a number of different geometric
situations where a relevant function, of geometric content, obeys an inequality of type
(\ref{inequality}). For instance, this is the case if $\psi=\left\vert
df\right\vert $ where $f:M\rightarrow N$ is a harmonic map or $\psi=\left\vert
Ric\right\vert $ where $Ric$ is the Ricci tensor of a locally conformally flat
manifold $M$ with zero scalar curvature. However, in many geometric
applications, this function is required to satisfy (scale-invariant)
integral conditions that cannot be captured in the above theorem, unless more
restrictive spectral assumptions are imposed. Thus, we are led to ask the following

\begin{problem*}
\label{problem}Is there an abstract formulation of Theorem \ref{th_CSZ}, i.e.,
a vanishing result for solutions of (\ref{inequality}) with higher
integrability exponents?
\end{problem*}

\section{Statement of the new abstract result}

Note that, in the special situation of minimal submanifolds, it happens that
the solution to the PDE, the potential term of the Schr\"{o}dinger (stability)
operator, the bound of the curvature tensor (Gauss equations) etc... coincide.
This is the reason why the result looks so elegant: one  $L^{p}$ assumption
controls  so many different objects! Furthermore, it is well known that
Euclidean minimal submanifolds enjoy an isoperimetric (hence Sobolev)
inequality and much more. In particular, volumes are at least Euclidean.
Obviously all this is, in general, no longer true in the abstract setting and some
choices have to be made. In this note we propose the following answer to
Problem 1.

\begin{theorem}
\label{th_abstract}Let $\left(  M,\left\langle ,\right\rangle \right)  $ be a
complete, $m(\geq 3)$-dimensional Riemannian manifold supporting the Euclidean Sobolev inequality%
\[
\left\Vert v\right\Vert _{L^{\frac{2m}{m-2}}}\leq S\left\Vert \nabla
v\right\Vert _{L^{2}}\text{,}%
\]
for every $v\in C_{c}^{\infty}\left(  M\right)  $ and for some constant $S>0$.
Assume also that%
\[
\mathrm{vol}\left(  B_{R}\right)  =O\left(  R^{m}\right)  .
\]
Let $0\leq\psi\in Lip_{loc}$ be a distributional solution of the differential
inequality (\ref{inequality}) where $0\leq a\left(  x\right)  \in L^{m/2}$
satisfies the uniform decay estimate%
\[
\sup_{M\backslash B_{R}}a\left(  x\right)  =O\left(  \frac{1}{R^{2}}\right)
.
\]
Suppose that%
\[
\psi\in L^{2p}\left(  M\right)  ,
\]
for some $p>(1+A)/2$. Then:

\begin{enumerate}
\item[(i)] If the spectral condition (\ref{spectral}) holds for some%
\[
H>1+A>0
\]
then either $a  \equiv0$ or $\psi\equiv0$.

\item[(ii)] If the spectral condition (\ref{spectral}) holds with%
\[
H=1+A>0,
\]
then $\psi$ satisfies (\ref{inequality}) with the equality sign.
\end{enumerate}
\end{theorem}

In geometric situations, the potential term $a\left(  x\right)  $ has a
geometric content (the norm of the second fundamental form, of a curvature
tensor etc.) and satisfies some semi-linear elliptic inequality. It
follows that the uniform decay estimate required on $a\left(  x\right)  $ is a
consequence of the assumption $a\in L^{m/2}$; see e.g. \cite{CH-Commentarii}, \cite{PV-DJA},
\cite{V-uniform} and references therein.

It is well known that, in the presence of a Sobolev inequality, the spectral
assumption is implied by the requirement that the potential term $a\left(
x\right)  $ is small in a suitable integral sense when compared with the
Sobolev constant. Using this point of view, the vanishing part of Theorem
\ref{th_abstract} can be formulated in the following elegant form.

\begin{corollary}
Let $\left(  M,\left\langle ,\right\rangle \right)  $ be a complete, $m(\geq 3)$-dimensional Riemannian
manifold supporting the Euclidean Sobolev inequality%
\[
\left\Vert v\right\Vert _{L^{\frac{2m}{m-2}}}\leq S\left\Vert \nabla
v\right\Vert _{L^{2}}\text{,}%
\]
for every $v\in C_{c}^{\infty}\left(  M\right)  $ and for some constant $S>0$.
Assume also that%
\[
\mathrm{vol}\left(  B_{R}\right)  \leq O\left(  R^{m}\right)  .
\]
Let $0\leq\psi\in Lip_{loc}$ be a distributional solution of the differential
inequality (\ref{inequality}) where $0\leq a\left(  x\right)  \in L^{m/2}$
satisfies%
\[
\sup_{M\backslash B_{R}}a\left(  x\right)  =O\left(  \frac{1}{R^{2}}\right)
\]
and%
\[
\left\Vert a\right\Vert _{L^{\frac{m}{2}}}\leq\left(  HS^{2}\right)  ^{-1},
\]
for some%
\[
H>1+A>0.
\]
If $\psi\in L^{2p}$ for some $p>(1+A)/2$, then either $a\equiv0$ or $\psi\equiv0$.
\end{corollary}

We note that, actually, the upper bound on $\left\Vert a\right\Vert
_{L^{\frac{m}{2}}}$ can be improved bypassing Theorem \ref{th_vanishing1} and
using directly the Sobolev inequality as explained in \cite{PRS-TAMS},
\cite{PRS-Book}. We leave the corresponding statement to the interested reader.

\section{Back to minimal hypersurfaces}

Before proceeding in our analysis, let us observe how Theorem \ref{th_CSZ}
follows from the vanishing part of Theorem \ref{th_abstract}. Actually, it
improves the original version. This generalization of Shen-Zhou result has
been recently observed also by H.\ Fu and Z. Li, \cite{FZ-IJM}. However, their
approach is less direct than ours and relies on a structure theorem by
Anderson, \cite{And-preprint}. Accordingly, they show that, in the above
assumptions, the minimal submanifold has only one end; see also Section
\ref{section analysis}. Recall that the immersed, minimal submanifold is said to
be $H$-stable if the \textquotedblleft modified stability
operator\textquotedblright\ $\mathcal{L}_{H}=-\Delta-H\left\vert
\mathrm{II}\right\vert ^{2}$ satisfies $\lambda_{1}\left(  \mathcal{L}%
_{H}\right)  \geq0$.

\begin{corollary}
Let $f:M^{m}\rightarrow\mathbb{R}^{N}$ be a complete, $m$-dimensional, minimal
submanifold of finite total curvature $\left\vert \mathrm{II}\right\vert
^{2}\in L^{m/2}\left(  M\right)  $, $m\geq3$. If $M$ is $H$-stable for some
$H>\left(  m-2\right)  /m$ then $f\left(  M\right)  $ is an affine $m$-plane.
\end{corollary}

\begin{proof}
By the Simons equation, $\psi=\left\vert \mathrm{II}\right\vert $ satisfies
(\ref{inequality}) with%
\[
a\left(  x\right)  =\left\vert \mathrm{II}\right\vert ^{2}\in L^{m/2}\text{
and }A=-\left(  m+2\right)  /m.
\]
Anderson curvature estimates (see also \cite{PV-DJA}) state that%
\[
\left\vert \mathrm{II}\right\vert ^{2}=o\left(  \frac{1}{R^{2}}\right)  \text{
as }R\rightarrow+\infty.
\]
Since $0<1+A=\left(  m-2\right)  /m$, $\psi\in L^{m>1}$, and $\lambda
_{1}\left(  \mathcal{L}_{H}\right)  \geq0,$ for some $H>\left(  m-2\right)  /m$ then, by Theorem \ref{th_abstract} (i), either $a=\left\vert \mathrm{II}\right\vert ^{2}\equiv0$ or $\psi=\left\vert
\mathrm{II}\right\vert \equiv0$. In any case, the submanifold is planar.
\end{proof}

\begin{remark}
The geometric setting of minimal submanifolds allows us to
test the sharpness of our abstract theorem. In the recent paper \cite{TZ-TAMS}
it is shown that the $m$-dimensional catenoid is a (non-planar) minimal
hypersurface of $\mathbb{R}^{m+1}$ satisfying $\lambda_{1}\left(
\mathcal{L}_{H}\right)  \geq0$ with $H=\left(  m-2\right)  /m$. A direct
computation also shows that its total curvature is finite. Noting
that if  $a\left(  x\right)
\geq0$ and $H_{1}>H_{2}$ then $\lambda_{1}\left(  \mathcal{L}_{H_{2}}\right)
\geq\lambda_{1}\left(  \mathcal{L}_{H_{1}}\right)  $, it follows from the discussion
in the previous proof, that the condition
$H>\left(  1+A\right)  $ in Theorem \ref{th_abstract} cannot be relaxed.
\end{remark}

As a straighforward application of the non-vanishing part of Theorem
\ref{th_abstract} we get the following characterization of the higher
dimensional catenoids.
\begin{corollary}
Let $f:M^{m}\rightarrow\mathbb{R}^{m+1}$ be a complete, non-planar, minimal
hypersurface of finite total curvature $\left\vert \mathrm{II}\right\vert
^{2}\in L^{m/2}\left(  M\right)  $, $m\geq3$. If $M$ is $\frac{m-2}{m}$-stable
then it is a catenoid.
\end{corollary}

\begin{proof}
According to a nice result in \cite{TZ-TAMS} it suffices to show that%
\[
\left\vert \mathrm{II}\right\vert \left(  \Delta\left\vert \mathrm{II}%
\right\vert +\left\vert \mathrm{II}\right\vert ^{2}\left\vert \mathrm{II}%
\right\vert \right)  =\frac{2}{m}\left\vert \nabla\left\vert \mathrm{II}%
\right\vert \right\vert ^{2}.
\]
But this follows immediately from Theorem \ref{th_abstract} (ii) applied with
the choices $a\left(  x\right)  =\left\vert \mathrm{II}\right\vert ^{2}\in
L^{m/2}$, $A=-\left(  m+2\right)  /m$, $H=A+1=\left(  m-2\right)  /m$ and
$\psi=\left\vert \mathrm{II}\right\vert .$
\end{proof}

We note that the same conclusion can be deduced e.g. by \cite{CZ-CAG}
where the authors prove a structure theorem for $\frac{m-2}{m}$-stable minimal
hypersurfaces with bounded second fundamental form. More precisely, they show that, in these assumptions,
the hypersurface is a catenoid provided it has at least two ends. Since, by Anderson, non-planar minimal hypersurfaces with finite total curvature
have bounded second fundamental form and more than one end, the asserted characterization follows. As the above proof shows, our argument is much more straightforward.

\section{Analysis of the proofs of the geometric
theorem\label{section analysis}}

In order to prove Theorem \ref{th_CSZ}, the arguments supplied by
Schoen-Simon-Yau, do Carmo-Peng and B\'{e}rard on the one hand, and Anderson
and Shen-Zhu on the other hand are very different. More precisely, both start
from uniform decay estimates for $\left\vert \mathrm{II}\right\vert $ and use
upper volume estimates in a crucial way. The decay estimates follow from
Moser-iteration arguments once it is observed that $v=\left\vert
\mathrm{II}\right\vert ^{2}$ satisfies the Simons (semi-linear) inequality
$\Delta v+2v^{2}\geq0$. Thus $\sup_{M\backslash B_{R}}\left\vert
\mathrm{II}\right\vert ^{2}=o\left(  R^{-2}\right)  $. Granted this, the
former approach is very analytic and abstract in nature. Using directly, and
from the very beginning, the spectral assumption in combination with equation
(\ref{simons}) and suitable test functions, it produces an estimate of the
$L^{m/2}$-norm of $\left\vert \mathrm{II}\right\vert ^{2}$ over large balls.
In contrast, the second approach is very geometric and relies on smooth
convergence of Riemannian manifolds. The spectral assumption appears only at
the final step and in relation with the $L^{1}$-norm of $\left\vert
\mathrm{II}\right\vert ^{2}$. More precisely, Anderson deduces from $(i)$ of
Theorem \ref{th_CSZ} that $M$ is planar provided it has only one end and this
is the case because of condition $(ii)$. Indeed, the potential theoretic
characterization of the ends by H.-D. Cao, Y. Shen and S. Zhu, \cite{CSZ-MRL}%
, shows that all the ends of $M$ are non-parabolic so that, by harmonic
function theory, \cite{LT-JDG}, \cite{LW-MRL}, condition $(ii)$ of Theorem
\ref{th_CSZ} forces $M$ to have only one end; see also \cite{PRS-Book}. Unlike
Anderson, Shen-Zhu show that $(i)$ of Theorem \ref{th_CSZ} forces a uniform
decay of $\left\vert \mathrm{II}\right\vert ^{2}$ that is faster than
expected, i.e., $\sup_{M\backslash B_{R}}\left\vert \mathrm{II}\right\vert
^{2}=O\left(  R^{-m}\right)  $. This is obtained via a theorem by Schoen on the curvature decay of minimal graphs of
bounded slope. Now, as established by Anderson, $M$ has Euclidean volume
growth. It follows that $\left\vert \mathrm{II}\right\vert ^{2}\in
L^{1}\left(  M\right)  $ and the conclusion is reduced on the $L^{2}%
$-vanishing result by do Carmo-Peng, \cite{DCP-Beijing}.

\section{Proof of the abstract result\label{section proof}}

In a certain sense, the strategy by Shen-Zhu looks promising. According to
Theorem \ref{th_vanishing1}, on noting also that Moser-type arguments work for
general PDEs, what we need is a way to improve the uniform decay of the
solution of (\ref{inequality}). It is reasonable that this can be obtained
thanks to the presence of the gradient term, that corresponds to a refined
Kato inequality. But this kind of arguments are already known in an ambient
that only apparently is far from the minimal surface theory. Namely, the
theory of ALE ends of conformally flat, half-conformally or, more generally,
Bach-flat $4$-manifolds, see \cite{BKN-Invent}, and e.g. \cite{CH-Commentarii}%
, \cite{TV1-Invent}, \cite{TV-CommMathHelv}. In fact, we have the following result by S. Bando, A.
Kasue and H. Nakajima; see Section 4 in \cite{BKN-Invent}.

\begin{theorem}
\label{th_BKN}Let $\left(  M,\left\langle ,\right\rangle \right)  $ be a
complete, $m(\geq 3)$-dimensional Riemannian manifold supporting the Euclidean Sobolev inequality%
\[
\left\Vert v\right\Vert _{L^{\frac{2m}{m-2}}}\leq S\left\Vert \nabla
v\right\Vert _{L^{2}}\text{,}%
\]
for every $v\in C_{c}^{\infty}\left(  M\right)  $ and for some constant $S>0$.
Assume also that%
\[
\mathrm{vol}\left(  B_{R}\right)  =O\left(  R^{m}\right)  .
\]
Let $0\leq u\in L^{2p},$ $p>1/2$, be a weak solution of the differential
inequality%
\[
\Delta u+a\left(  x\right)  u\geq0.
\]
If $a\in L^{m/2}\cap L^{q}$ for some $q>m/2$ and%
\[
\int_{M\backslash B_{R}}a^{q}=O\left(  \frac{1}{R^{2q-m}}\right)  ,
\]
then%
\[
\sup_{M\backslash B_{R}}u=O\left(  \frac{1}{R^{\alpha}}\right)  ,
\]
for every $\alpha<m-2$.
\end{theorem}

Now, suppose we are in the assumptions of Theorem \ref{th_abstract} so that
$\psi$ satisfies inequality (\ref{inequality}) and $\psi\in L^{2p}$ for some
$p>(1+A)/2$. As we have already recalled during the proof of Theorem \ref{th_direct}, we have that $\psi^{1+A} \in W^{1,2}_{loc}$; \cite{PRS-Revista}, \cite{PRS-Book}. Therefore, using the test function $(\psi + \varepsilon)^{A-1}\rho$, $0 \leq \rho \in Lip_{c}(M)$, in the weak formulation of (\ref{inequality}) and letting $\varepsilon \to 0 $ yields that $u=\psi^{1+A}$ satisfies
\begin{equation}
\Delta u+bu\geq0, \label{formal}%
\end{equation}
weakly on $M$, where we have set $b\left(  x\right)  =\left(  1+A\right)  a\left(  x\right)
\geq0$; see page 209 in \cite{PRS-Book}. Since $b\in L^{m/2}$ and%
\[
\sup_{M\backslash B_{R}}b=O\left(  \frac{1}{R^{2}}\right)  ,
\]
we deduce that, for every $q>m/2$,%
\begin{align*}
\int_{M\backslash B_{R}}b^{q}  &  \leq\sup_{M\backslash B_{R}}b^{q-m/2}%
\left\Vert b\right\Vert _{L^{m/2}}^{m/2}\\
&  =O\left(  \frac{1}{R^{2q-m}}\right)  .
\end{align*}
Therefore, we can use Theorem \ref{th_BKN} to deduce%
\[
\sup_{M\backslash B_{R}}u=O\left(  \frac{1}{R^{\alpha}}\right)  ,
\]
for every $\alpha<m-2.$ This means that the original solution $\psi$ of
(\ref{inequality}) satisfies%
\[
\sup_{M\backslash B_{R}}\psi=O\left(  \frac{1}{R^{\frac{m-2}{1+A}-\varepsilon
}}\right)  ,
\]
for every $0<\varepsilon<<1$. Thus, using the co-area formula, the volume
assumption, the fact that $H\geq A+1$, and integrating by parts,
\begin{align*}
\int_{B_{R}\backslash B_{1}}\psi^{2H}  &  =\int_{1}^{R}\int_{\partial B_{t}%
}\psi^{2H}\\
&  \leq C\int_{1}^{R}\frac{\mathrm{Area}\left(  \partial B_{t}\right)
}{t^{2H\left(  \frac{m-2}{1+A}\right)  -\varepsilon}}\\
&  \leq C\int_{1}^{R}\frac{\mathrm{Area}\left(  \partial B_{t}\right)
}{t^{2m-4-\varepsilon}}\\
&  \leq C_{1}\left\{  \frac{\mathrm{Vol}\left(  B_{R}\right)  }%
{R^{2m-4-\varepsilon}}+\int_{1}^{R}\frac{\mathrm{Vol}\left(  B_{t}\right)
}{t^{2m-3-\varepsilon}}+1\right\} \\
&  \leq C_{2}\left\{  R^{-m+4+\varepsilon}+\int_{1}^{R}t^{-m+3+\varepsilon
}+1\right\} \\
&  \leq C_{3}\left\{  R^{-m+4+\varepsilon}+1\right\}  .
\end{align*}
It follows that%
\[
\int_{B_{R}\left(  o\right)  }\psi^{2H}=o\left(  R^{2}\right)  \text{, as
}R\rightarrow+\infty\text{,}%
\]
for every $m\geq3$. Now, if $H>\left(  1+A\right)  $, since $\lambda
_{1}\left(  \mathcal{L}_{H}\right)  \geq0$, application of Theorem \ref{th_vanishing1} (i)
or Theorem \ref{th_direct} (i) yields that either $a\equiv0$ or $\psi\equiv0$. On the other hand, if
$H=\left(  1+A\right)  $ we can apply Theorem \ref{th_vanishing1} (ii) or Theorem \ref{th_direct} (ii)
to deduce that $\psi$ satisfies (\ref{inequality}) with the equality
sign.
\section{Further applications}

\subsection{Topology of $H$-stable minimal hypersurfaces}

As observed in the introduction, the \textquotedblleft
direct\textquotedblright\ approach has been proposed to face the problem of
the triviality of minimal surfaces with finite total curvature,
\cite{Berard2-Remarque}. We also pointed out that this method is not
suitable to obtain Theorem \ref{th_CSZ} for higher dimensions $m$ because an
upper bound for the integrability exponent arises. Nevertheless, since it
permits to consider the case $H<p$, it reveals useful to obtain topological
informations on minimal surface immersed in non-negatively curved manifolds
without integrability assumptions on $|\mathrm{II}|$, provided $M$ is
$H$-stable. Namely, in the spirit of \cite{SY-CH}, we obtain the following

\begin{theorem}
\label{topo} Let $M^{m}$ be a complete non-compact minimally immersed
hypersurface in a manifold of non-negative sectional curvature. Suppose that
$M$ is $H$-stable, for some $H>\left(  m-1\right)  /m$. If $D\subset M$ is a
compact domain in $M$ with smooth, simply connected boundary, then there is no
non-trivial homomorphism of $\pi_{1}(D)$ into the fundamental group of a
compact manifold with non-positive sectional curvature.
\end{theorem}

The proof relies on a suitable use of harmonic maps; see also Remark 6.22 in
\cite{PRS-Book}. Consider a complete Riemannian manifold $M$ whose Ricci
curvature satisfies
\begin{equation}
{}^{M}Ric\geq-k(x), \label{ric}%
\end{equation}
for some continuous function $k\geq0$ and let $h:M\rightarrow N$ be a harmonic
map with finite energy $|dh|\in L^{2}$. In $N$ has non-positive curvature,
then the energy density satisfies the Bochner inequality%
\[
\frac{1}{2}\Delta\left\vert dh\right\vert ^{2}+k\left(  x\right)  \left\vert
dh\right\vert ^{2}\geq\left\vert Ddh\right\vert ^{2}.
\]
Furthermore, the following refined Kato inequality holds true, \cite{Br},
\cite{CGH},%
\[
|Ddh|^{2}\geq\frac{m}{m-1}|\nabla|dh||^{2}.
\]
Therefore $\psi=\left\vert dh\right\vert $ satisfies (\ref{inequality}) with
$A=-1/\left(  m-1\right)  $. Since every map $f:M\rightarrow N$ with finite
energy $\left\vert df\right\vert \in L^{2}$ is homotopic to a harmonic map
$h:M\rightarrow N$ of finite energy, applying to $h$ the vanishing result of
Theorem \ref{th_direct}, we get

\begin{proposition}
\label{harmonic} Let $f:M\rightarrow N$ be a continuous map from a complete, $m$-dimensional
manifold $\left(  M,\left\langle ,\right\rangle \right)  $ with Ricci
curvature satisfying (\ref{ric}) into a compact manifold of non-positive
sectional curvature. Assume that $f$ has finite energy $\left\vert
df\right\vert ^{2}\in L^{1}\left(  M\right)  $. If
\[
\lambda_{1}\left(  -\Delta-Hk\left(  x\right)  \right)  \geq0,
\]
for some
\[
H>(m-1)/m,
\]
then $f$ is homotopic to a constant.
\end{proposition}

Now, suppose that $M$ is isometrically immersed as complete, $H$-stable,
minimal hypersurface into a space $Q$ with ${}\operatorname{Sect}_{Q}\geq0$.
According to Gauss equations, ${}^{M}Ric\geq-|\mathrm{II}|^{2}%
$. Moreover, by assumption, the operator $\mathcal{L}_{H}=-\Delta-H|\mathrm{II}|^{2}$
satisfies $\lambda_{1}(\mathcal{L}_{H})\geq0$. Hence, we are precisely in the
assumptions of Proposition \ref{harmonic}, thus obtaining that each harmonic
map with finite energy from $M$ into a non-positively curved $N$ is homotopic
to a constant. Starting from this point, we can conclude the proof of Theorem
\ref{topo} proceeding as in \cite{SY-CH}. See also \cite{PRS-Book}, Theorem 6.21.

\subsection{Locally conformally flat manifolds}

As expected, Theorem \ref{th_abstract} works perfectly in the setting of
conformally flat manifolds. Recall that the $m$-dimensional Riemannian
manifold $\left(  M^{m},\left\langle ,\right\rangle \right)  $ is said to be
locally conformally flat if every point of $M$ has a neighborhood which is
conformally immersed into the standard sphere $\mathbb{S}^{m}$. We have the
following result which, in our opinion, fits very well in the nice works by G.
Carron and M. Herzlich, \cite{CH-Commentarii}, and, G. Tian and J. Viaclovsky,
\cite{TV1-Invent}, \cite{TV-CommMathHelv}.

\begin{theorem}
Let $\left(  M,\left\langle ,\right\rangle \right)  $ be a complete, simply
connected, scalar flat, locally conformally flat Riemannian manifold of
dimension $m\geq3$. Assume that $\left\vert Ric\right\vert \in L^{m/2}$ and
that the Schr\"{o}dinger operator%
\[
\mathcal{L}_{H}=-\Delta-H\sqrt{\frac{m}{m-1}}\left\vert Ric\right\vert
\]
satisfies $\lambda_{1}\left(  \mathcal{L}_{H}\right)  \geq0$ for some
$H\geq\left(  m-2\right)  /m$. Then:

\begin{enumerate}
\item[(i)] If $H>\left(  m-2\right)  /m$ then $M=\mathbb{R}^{m}$.

\item[(ii)] If $H=\left(  m-2\right)  /m$ then, either $M=\mathbb{R}^{m}$ or
$\left\vert Ric\right\vert >0$ and, in a neighborhood $U$ of each point
$x\in\left\{  DRic\neq0\right\}  \neq\emptyset$, there is an isometric
splitting $U=\left(  -\varepsilon,\varepsilon\right)  \times_{f}N$ where $N$
has constant curvature $K$ and $f$ satisfies%
\[
\left(  m-2\right)  \left(  f^{\prime}\right)  ^{2}+2ff^{\prime\prime
}-K\left(  m-2\right)  =0.
\]

\end{enumerate}
\end{theorem}

\begin{proof}
Our basic reference for conformally flat manifolds is Chapter 6 in Schoen and
Yau book \cite{SY-Book}. See also Section 9 in \cite{PRS-Book} for the
relevant PDEs involving the traceless Ricci tensor of a conformally flat manifold.

It is a classical result by N. Kuiper that the simply connected, locally
conformally flat $M^{m}$ has a conformal immersion into $\mathbb{S}^{m}$.
Therefore, according to Schoen and Yau, the Yamabe constant of $M$ satisfies
$Q\left(  M\right)  =Q\left(  \mathbb{S}^{m}\right)  >0$. Since $M$ is
scalar-flat, this means that $M$ enjoys the Sobolev inequality%
\begin{equation}
\left\Vert v\right\Vert _{L^{\frac{2m}{m-2}}}\leq S\left\Vert \nabla
v\right\Vert _{L^{2}},\label{confflat-sobolev}%
\end{equation}
with%
\[
S=\sqrt{\frac{4\left(  m-1\right)  }{Q\left(  \mathbb{S}^{m}\right)  \left(
m-2\right)  }}.
\]
Now, the norm of the Ricci tensor of $M$ satisfies the Simons-type identity%
\[
\frac{1}{2}\Delta\left\vert Ric\right\vert ^{2}=\left\vert DRic\right\vert
^{2}+\frac{m}{m-2}\mathrm{tr}\left(  Ric^{\left(  3\right)  }\right)  ,
\]
where $Ric^{\left(  3\right)  }$ denotes the third composition power of the
Ricci tensor. Moreover, since $Ric$ is scalar flat we have the classical
Okumura inequality,%
\[
\mathrm{tr}\left(  Ric^{\left(  3\right)  }\right)  \geq-\frac{m-2}%
{\sqrt{m\left(  m-1\right)  }}\left\vert Ric\right\vert ^{3},
\]
the equality holding at some point $x$ if and only if
$(m-1)$ eigenvalues of $Ric\left(  x\right)  $ coincide,
\cite{AdC-PAMS}. Finally, since $Ric$ is a Codazzi tensor the refined Kato
inequality%
\[
\left\vert DRic\right\vert ^{2}\geq\frac{m+2}{m}\left\vert \nabla\left\vert
Ric\right\vert \right\vert ^{2},
\]
holds. Summarizing, we have%
\begin{equation}
\left\vert Ric\right\vert \left(  \Delta\left\vert Ric\right\vert +\sqrt
{\frac{m}{m-1}}\left\vert Ric\right\vert ^{2}\right)  \geq\frac{2}%
{m}\left\vert \nabla\left\vert Ric\right\vert \right\vert ^{2}%
,\label{confflat1}%
\end{equation}
weakly on $M$. In particular, $u=\left\vert Ric\right\vert $ satisfies the
semilinear elliptic inequality%
\[
\Delta u+\sqrt{m/\left(  m-1\right)  }u^{2}\geq0.
\]
When combined with the Sobolev inequality this yields that $\left\vert
Ric\right\vert =o\left(  R^{-2}\right)  $ as $R\rightarrow+\infty$ (see, e.g.
\cite{CH-Commentarii}, Lemma 5.2). Since $M$ is locally conformally flat and
scalar flat, by the decomposition of the Riemann tensor we also have
$\left\vert Riem\right\vert =o\left(  R^{-2}\right)  $. It follows from the
volume growth estimates by Tian-Viaclovsky that $\mathrm{vol}\left(
B_{R}\right)  =O\left(  R^{m}\right)  $. Therefore, if $H>\left(  m-2\right)
/m$, we can apply Theorem \ref{th_abstract} (i) with the choices
$\psi=\left\vert Ric\right\vert $, $a\left(  x\right)  =\sqrt{m/\left(
m-1\right)  }\left\vert Ric\right\vert $, $A=-2/m$ to conclude that either
$a\equiv0$ or $\psi\equiv0$. In any case $Ric=0$ which, in turn, forces
$Riem=0$. The desired conclusion $M=\mathbb{R}^{m}$ then follows from the Hopf
classification theorem.
On the other hand, suppose $H=\left(  m-2\right)  /m$.
Assume also that $M\neq\mathbb{R}^{m}$, for otherwise there is nothing to
prove. Application of Theorem \ref{th_abstract} (ii) \ yields the equality in
(\ref{confflat1}), i.e.,%
\begin{equation}
\left\vert Ric\right\vert \left(  \Delta\left\vert Ric\right\vert +\sqrt
{\frac{m}{m-1}}\left\vert Ric\right\vert ^{2}\right)  =\frac{2}{m}\left\vert
\nabla\left\vert Ric\right\vert \right\vert ^{2}.\label{conflat2}%
\end{equation}
From the equality case in Okumura inequality we obtain that $Ric$ has
two (possibly equal) eigenvalues $\mu$ and $-\mu/\left(  m-1\right)  $ of multiplicity $1$ and
$\left(  m-1\right)  ,$ respectively. We claim that they are everywhere
distinct. Indeed, using (\ref{conflat2}) and arguing as in Section
\ref{section proof}, we see that $v=\left\vert Ric\right\vert ^{\left(
m-2\right)  /m}$ is a weak solution of%
\[
\Delta v+b\left(  x\right)  v=0,
\]
where $b\left(  x\right)  =(1-2/m)\sqrt{m/\left(  m-1\right)  }\left\vert
Ric\right\vert $. It follows from the strong minimum principle that either
$v>0$ or $v\equiv0$. The second possibility cannot occur because, in this
case, $Riem=0$ and  $M=\mathbb{R}^{m}$, against our assumption. Thus
$\left\vert Ric\right\vert >0$ and this forces $\mu\neq-\mu/\left(
m-1\right)  $, as claimed.

Now, we observe that $Ric$ is\ not parallel. Indeed, suppose the contrary.
Since $\left\vert Ric\right\vert $ is constant, $\left\vert Ric\right\vert \in
L^{m/2}$ and, due to the Sobolev inequality (\ref{confflat-sobolev}),
$\mathrm{vol}\left(  M\right)  =+\infty$, then $Ric=0$. As above, this implies
$M=\mathbb{R}^{m}$, contradicting our initial assumption. Therefore,  $Ric$ is
not parallel. Since $Ric$ is a Codazzi tensor with constant (zero) trace and
two distinct eigenvalues, we can apply a result by A. Derdzinski, \cite{De},
and conclude that every point $x\in\left\{  DRic\neq0\right\}  $ has a
neighborhood of the form $U=\left(  -\varepsilon,\varepsilon\right)
\times_{f}\Sigma$, with $f$ nonconstant. Since $U$ is locally conformally flat
then the $\left(  m-1\right)  $-dimensional manifold $\Sigma$ must be of
constant curvature, \cite{BGV}, say $^{\Sigma}Sec=K$. Computing the scalar
curvature of the warped product $U$, and recalling that $M$ is scalar flat, we
get%
\[
\left(  m-2\right)  \left(  f^{\prime}\right)  ^{2}+2ff^{\prime\prime
}-K\left(  m-2\right)  =0,
\]
thus completing the proof.
\end{proof}

\subsection*{Acknowledgement}
The authors are grateful to A.G. Setti for many stimulating conversations and for several important remarks. They would also like to thank the anonymous referee for having pointed out some oversights in the proof of Theorem \ref{th_direct}.

\end{document}